\documentclass{amsart}

\usepackage{amsmath, amssymb, amsthm}
\usepackage{graphicx}
\usepackage{color}
\usepackage{enumerate}
\usepackage{cite}
\usepackage{tabu}

\usepackage{url}
\usepackage{hyperref}

\usepackage{fullpage}
\usepackage[active]{srcltx}

\usepackage{chngcntr}
\usepackage{apptools}
\AtAppendix{\counterwithin{theorem}{section}}

\usepackage{multibib}

\usepackage{tikz}
\usetikzlibrary{arrows}

\newcommand{\myrule} [3] []{
    \begin{center}
        \begin{tikzpicture}
            \draw[#2-#3,  thick, #1] (0,0) to (0.35\linewidth,0);
        \end{tikzpicture}
    \end{center}
}

\theoremstyle{plain}
\newtheorem{theorem}{Theorem} 

\newtheorem{corollary}[theorem]{Corollary}
\newtheorem{conj}[theorem]{Conjecture}
\newtheorem{proposition}[theorem]{Proposition}
\newtheorem{lemma}[theorem]{Lemma}

\newtheorem{question}[theorem]{Question}

\newtheorem*{claim*}{Claim}

\theoremstyle{definition}

\newtheorem{remark}[theorem]{Remark}

\newcommand{\Z}{\ensuremath{{\mathbb{Z}}}}
\newcommand{\Q}{\ensuremath{{\mathbb{Q}}}}

\newcommand{\X}{\ensuremath{{\widetilde X}}}
\renewcommand{\P}{\ensuremath{{\mathcal P}}}

\newcommand{\HF}{\ensuremath{{\widehat{\rm HF}}}}
\newcommand{\HFK}{\ensuremath{{\widehat{\rm HFK}}}}

\newcommand{\bdry}{\ensuremath{\partial}}

\DeclareMathOperator{\Rk}{rk}

\DeclareMathOperator{\vol}{vol}

\newcommand{\nbhd}{\ensuremath{\mathcal{N}}}

\title[The Poincar\'e homology sphere, lens space surgeries]{The Poincar\'e homology sphere, lens space surgeries, and some knots with tunnel number two.}

\author[Kenneth L.\ Baker and Neil R.\ Hoffman]{Kenneth L.\ Baker \\ Appendix by Neil R.\ Hoffman}

\address{Department of Mathematics, University of Miami, 
Coral Gables, FL 33146, USA}
\email{k.baker@math.miami.edu}

\address{Department of Mathematics, Oklahoma State University, Stillwater OK,  74078-1058, USA}
\email{neil.r.hoffman@okstate.edu}

\begin{document}

\begin{abstract}
We exhibit an infinite family of knots in the Poincar\'e homology sphere with tunnel number $2$ that have a lens space surgery.  Notably, these knots are not doubly primitive and provide counterexamples to a few conjectures.  
Additionally, we update (and correct) our earlier work on Hedden's almost-simple knots.
In the appendix, it is shown that a hyperbolic knot in the Poincar\'e homology sphere with a lens space surgery has either no symmetries or just a single strong involution.
\end{abstract}

\maketitle

\begin{figure}
\centering
\includegraphics[width=5in]{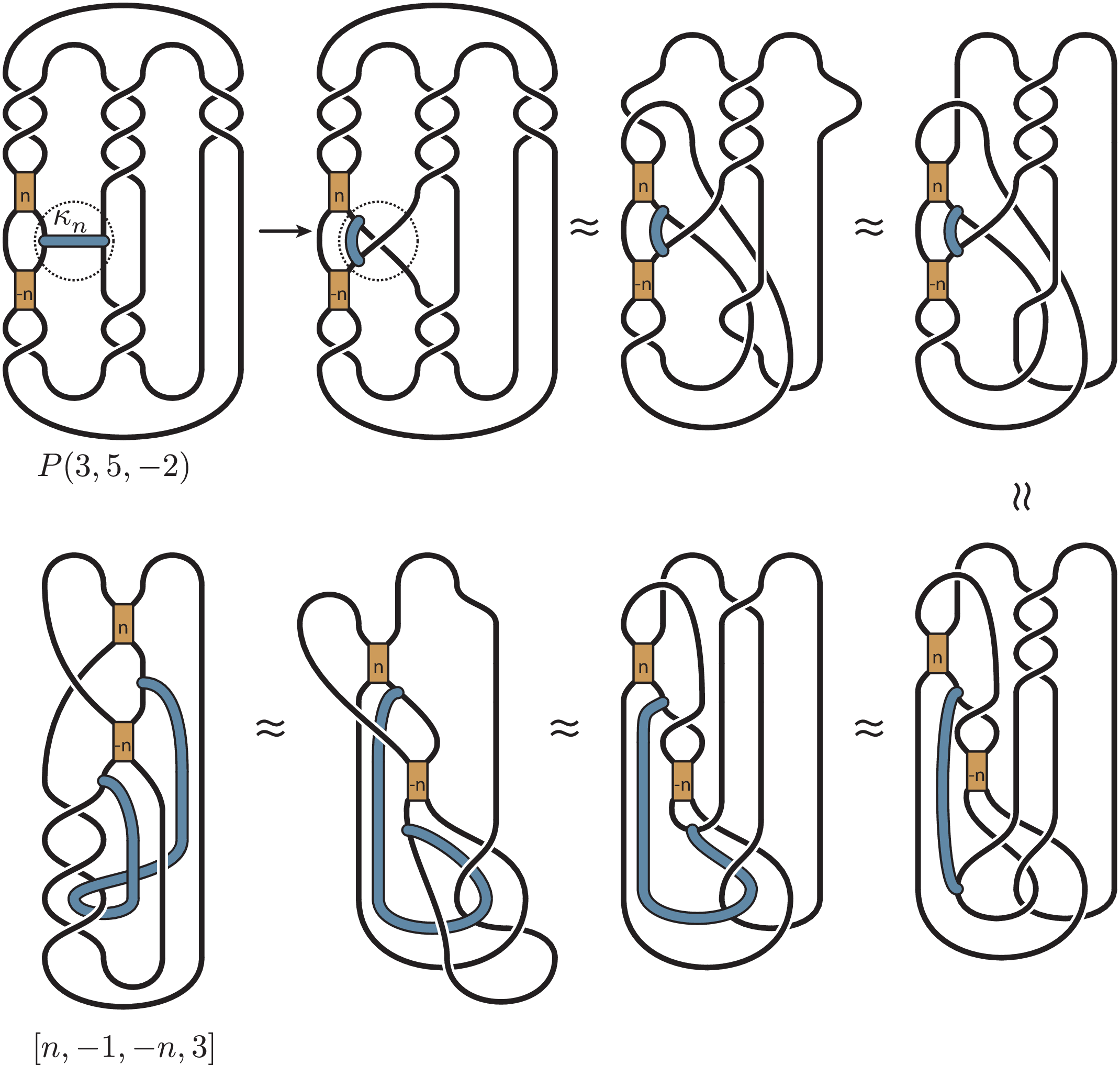}
\caption{Top left shows a family of arcs $\kappa_n$ with end points on the pretzel knot $P(-2,3,5)$.  A numbered vertical box signifies a stack of half-twists where the sign informs the handedness.   We perform a  banding on $\kappa_n$ (framed relative to the horizontal bridge sphere) followed by a sequence of isotopies to recognize the result as the two-bridge knot $B(3 n^2 +  n + 1,  -3n+2)$. The  arc $\kappa_n^*$ dual to the banding is also carried along.}
\label{fig:P-235to2bridge}
\end{figure}

\section{A family of knots in the Poincar\'e homology sphere with lens space surgeries.}

	Figure~\ref{fig:P-235to2bridge} shows a one-parameter family of arcs $\kappa_n$, $n \in \Z$, on the pretzel knot $P(-2,3,5)=P(3,5,-2)$ for which a $-1$--framed banding produces the two bridge knot $B(3 n^2 + n + 1, -3n+2)$. Indeed, $[n,-1,-n,3] = -\tfrac{3 n^2 +  n + 1}{ -3n+2}$.  Through double branched coverings and the Montesinos Trick, the arc $\kappa_n$ lifts to a knot $K_n$ in the Poincar\'e homology sphere on which an integral surgery yields the lens space $L(3n^2+n+1, -3n+2)$.

As presented, the arc $\kappa_n$ lies in a bridge sphere that presents $P(-2,3,5)$ as a $3$--bridge link.  Thus the knot $K_n$ lies in a genus $2$ Heegaard surface for the Poincar\'e homology sphere $\P$.  Nevertheless, as we shall see, generically $K_n$ does not have tunnel number one.

\begin{theorem}\label{thm:main}
For each integer $n$, the knot $K_n$ in the Poincar\'e homology sphere $\P$ has an integral surgery to the lens space $L(3n^2+n+1, -3n+2)$.  If $|n| \geq 4$, then $K_n$ has tunnel number $2$.
\end{theorem}

These knots arose from an application of the ``longitudinally jointly primitive'' ideas presented in \cite{bhl}.  After observing that the negative Whitehead link is a two-component genus $1$ fibered link for which $-1$ surgery on both components produces $\P$, we found a family of jointly primitive knots in the twice-punctured torus bundle that is the exterior of this Whitehead link.  Our knots $K_n$ are the images of these jointly primitive knots in the $\P$ filling.  This is discussed further in \cite[Section 2.3.3]{bhl}.

\smallskip

Previously, the only known knots in $\P$ admitting a surgery to a lens space are surgery duals to Tange knots \cite{tangeknots} and certain Hedden knots \cite{hedden, rasmussen, baker-almostsimple}.  (See Section~\ref{sec:heddensknots} for a discussion of the Hedden knots.) These knots in $\P$ are all {\em doubly primitive}, they may be presented as curves in a genus $2$ Heegaard splitting that represent a generator in $\pi_1$ of each handlebody bounded by the Heegaard surface.  We see this as Tange knots and Hedden knots all admit descriptions by doubly pointed genus $1$ Heegaard diagrams. (Tange knots are all simple knots in lens spaces and Hedden knots are ``almost simple''.)  Equivalently, they are all $1$--bridge knots in their lens spaces.  Any integral surgery on a $1$--bridge knot in a lens space naturally produces a $3$--manifold with a genus $2$ Heegaard surface in which the dual knot sits as a doubly primitive knot.

The Berge Conjecture posits that any knot in $S^3$ with a lens space surgery is a doubly primitive knot \cite{berge}, cf.\ \cite[Question 5.5]{gordonsurvey}.  It is also proposed that any knot in $S^1 \times S^2$ with a lens space surgery is a doubly primitive knot \cite[Conjecture 1.1]{bbls1xs2}, \cite[Conjecture 1.9]{greene}.  By analogy and due to a sense of simplicity\footnote{The manifolds $S^3$, $\P$, and its mirror $-\P$ are the only homology $3$--spheres with finite fundamental group (by Perelman) and the only known irreducible L-space homology $3$--spheres, e.g.\ \cite{efteSFS, eftetorus}.}, one may suspect that any knot in $\P$ with a lens space surgery should also be doubly primitive. However this is not the case.  Since our knots $K_n$ have tunnel number $2$  for $|n|\geq 4$, they cannot be doubly primitive.

\begin{corollary}
There are knots in the Poincar\'e homology sphere with lens space surgeries that are not doubly primitive.  That is, the analogy to the Berge Conjecture fails for the Poincar\'e homology sphere. \qed
\end{corollary}

\begin{corollary}
Conjecture 1.10 of \cite{greene} is false. For example, the lens space $L(5,1)$ contains a knot $K_1^*$ with a surgery to $\P$ but does not contain a Tange knot or a Hedden knot with a $\P$ surgery.
\end{corollary}

\begin{proof}
For $n=1$, surgery on $K_1$ produces $L(5,-1)$ as shown in Figure~\ref{fig:P-235to2bridge}.
A quick check of \cite[Table 2]{tangeknots} shows that $L(5,-1)$ contains no Tange knots.  The Hedden knots with surgery to $\pm\P$ are homologous to Berge lens space knots of type VII \cite{rasmussen, greene}: one further observes that $L(5,-1)$ does not contain any of these knots since $k^2+k+1=0\pmod{5}$ has no integral solution.
\end{proof}

\begin{remark}\label{rem:cable}
The knot $K_1$ is actually the $(2,3)$--cable of the knot $J$ that is surgery dual to the trefoil in $S^3$.
\end{remark}

As we shall see in Theorem~\ref{thm:properties}(3), it is the $(2g(K_n)-1)$--surgery on our knots $K_n$ that produces a lens space; $g(K_n)$ is the Seifert genus of the knot $K_n$.   Hence they also provide counterexamples to \cite[Conjecture 1.7]{hedden} (cf.\ \cite[Conjecture 1.6]{bgh}) and \cite[Conjecture 1.10]{greene}.  It seems plausible that it is the largeness of the knot genus with respect to the lens space surgery slope that enables the failure of our knots to be doubly primitive.   Hedden and Rasmussen also observe a distinction at this slope for lens space surgeries on knots in L-space homology spheres \cite{hedden, rasmussen}. 
 Indeed, from the view of Heegaard Floer homology for integral slopes on knots in homology spheres,  this slope is right at the threshold at which a knot could have an L-space surgery \cite[Proposition 9.6]{OSRational} (see also \cite[Lemma 2.13]{HeddenCabling2})\footnote{This threshold may be lower for knots in homology spheres with $\tau(K) < g(K)$.} and just below what implies it has simple knot Floer homology \cite{eaman}\footnote{A knot $K^*$ in a rational homology sphere $Y$ has {\em simple knot Floer homology} if $\Rk \HFK(Y,K^*) = \Rk \HF(Y)$.  This definition does not require $Y$ to be an L-space itself.}.   With this in mind, we adjust and update \cite[Conjecture 1.10]{greene}.

\begin{conj}
Suppose that $p$--surgery on a knot $K$ in the Poincar\'e homology sphere produces a lens space.   If $p > 2g(K)-1$ is an integer, then $K$ is a doubly primitive knot. Furthermore it is surgery dual to one of the Tange knots.
\end{conj}

\begin{conj}\label{conj:poincaredp}
If a knot in the Poincar\'e homology sphere is doubly primitive, then it is surgery dual to one of the Tange knots or one of the Hedden knots.
\end{conj}

As mentioned in the paragraph following its statement, \cite[Theorem 3.1]{hedden} (presented here in Theorem~\ref{thm:hedden3.1}) applies to any L-space homology sphere, not just $S^3$. Thus, since doubly primitive knots are surgery dual to one-bridge knots in lens spaces, \cite[Theorem 3.1]{hedden} implies that such surgery duals are either simple knots or one of the Hedden knots  $T_L$ or $T_R$.  In Section~\ref{sec:heddensknots} we clarify and correct our work in \cite{baker-almostsimple} and further classify when Hedden's knots are dual to integral surgeries on knots homology spheres, $\P$ in particular.  Tange has produced a list of simple knots in lens spaces that are surgery dual to knots in $\P$ \cite{tangeknots}, and it is expected that this list is complete.  Verifying its completeness will  affirm Conjecture~\ref{conj:poincaredp}.

\bigskip

Let us collect various properties of our knots.
\begin{theorem}\label{thm:properties}
Let $p=3n^2+n+1$ and $q=-3n+2$. Then the following hold for our family of knots $K_n$ in $\P$:
\begin{enumerate}
\item Positive $p$--surgery on $K_n$ produces a lens space $L(p,q)$.
\item $K_n$ is fibered and supports the tight contact structure on $\P$.
\item $g(K_n)=(p+1)/2$
\item $\Rk\HFK(L(p,q), K_n^*) = p+2$
\item Let $T_n$ be the $(3n+1,n)$--torus knot in $S^3$, and note that $p$--surgery on $T_n$ also produces $L(p,q)$.  The surgery dual to $K_n$ is homologous to the surgery dual to $T_n$.
\item $\Delta_{K_n}(t) = \Delta_{T_n}(t) - (t^{(p-1)/2} +  t^{-(p-1)/2})+(t^{(p+1)/2} +  t^{-(p+1)/2})$.
\end{enumerate}
\end{theorem}

Theorem~\ref{thm:properties}(5) is given as Lemma~\ref{lem:homologous}.  The remainder of Theorem~\ref{thm:properties} follows from assembling the works of Hedden \cite{hedden}, Rasmussen \cite{rasmussen}, and Tange \cite{tange-nonexistence,tange-ozszcorrectionterm}.  In fact, appealing to Greene's  proof of the Lens Space Realization Problem \cite{greene}, we tease out the following general theorem from which Theorem~\ref{thm:properties}(3)\&(6) follow.

\begin{theorem}\label{thm:largesurgery}
Suppose that $p$--surgery on knot $K$ in an L-space homology sphere $Y$ with $d(Y)=2$ produces the lens space $L(p,q)$.   
Then $p=2g(K)-1$ if and only if  $p$--surgery on some Berge knot $B$ in $S^3$ also produces the lens space $L(p,q)$ in which the surgery duals to $K$ and $B$ are homologous.

When this holds,   $\Delta_K(t) = \Delta_B(t) -(t^{(p-1)/2} +  t^{-(p-1)/2})+(t^{(p+1)/2} +  t^{-(p+1)/2})$.
\end{theorem}

%
%

\section{Questions about knots in the Poincar\'e homology sphere with lens space surgeries.}

\subsection{Homology classes of knots}
\begin{question}\label{ques1}
\quad
\begin{enumerate}
\item For which Berge knots $B$ that have a positive, odd $p$--surgery producing the lens space $L(p,q)$, is there a knot $K$ in $\P$ such that $2g(K)-1=p$, $p$--surgery produces the lens space $L(p,q)$, and the surgery duals $B^*$ and $K^*$ are homologous?

\item For a Berge knot $B$ as above, can there be two distinct knots $K_1$ and $K_2$ in $\P$ with $p$--surgery to $L(p,q)$ such that the surgery duals $B^*$, $K_1^*$, and $K_2^*$ are all homologous?
\end{enumerate}
\end{question}

Presently the only Berge knots known to answer Question~\ref{ques1} (1) are the Berge knots of type VII (those that embed in the fiber of the trefoil) due to \cite{hedden,baker-almostsimple} and the $(3n+1,n)$--torus knots with $p=3n^2+n+1$ due to our knots introduced here.   Indeed, what about the case of the $(3n+1,n)$--torus knots when $p=3n^2+n-1$?

To clarify the word `distinct' in Question~\ref{ques1} (2), note that the mapping class group of the Poincar\'e homology sphere is trivial  \cite[Th\`{e}or\'{e}m 3 \& Corollaire 4]{BoileauOtal}.  Thus knots in $\P$ that are related by a diffeomorphism of $\P$ are also isotopic.

\subsection{Symmetries}
One may observe that,  thus far, all the known examples of knots in $\P$ with a lens space surgery are strongly invertible.  Hence there is an involution of $\P$ taking each of these knots with some orientation to its reverse.   Wang-Zhou have shown that if a knot  in $S^3$ other than a torus knots has a non-trivial lens space surgery, then its only possible symmetry is a strong involution \cite{wangzhou}. 
\begin{question} \label{ques:symm}\quad 
\begin{enumerate}
\item Is every knot in $\P$ with a lens space surgery strongly invertible?
\item What are the possible symmetries of a knot in $\P$ with a lens space surgery?
\end{enumerate}
\end{question}
In the appendix to this article, Hoffman provides Theorem~\ref{thm:symm} which addresses Question~\ref{ques:symm}(2): A hyperbolic knot in $\P$ with a lens space surgery either has no symmetries or just a single strong involution.

\bigskip

\subsection{Hopf plumbings}
Let us inquire about another potential analogy with the Berge knots in $S^3$.
The Giroux Correspondence says that any two fibered knots supporting the same contact structure are related by a sequence of plumbings and de-plumbings of positive Hopf bands \cite{giroux}.
Since the Berge knots (with positive surgeries to lens spaces) can all be expressed as closures of positive braids\footnote{This is remarked preceding the Conjecture of \cite{godateragaito}, for example.}, each Berge knot can be obtained from the unknot by a sequence of plumbings of Hopf bands. No de-plumbings are necessary.

In $\P$ the unknot is not fibered, but the knot $J$ that is $-1$--surgery dual to the (negative) trefoil is.  As we shall see in Lemma~\ref{lem:gofk}, $J$ is a reasonable surrogate for the unknot in $\P$: it is the unique genus one fibered knot in $\P$, and it supports the tight contact structure on $\P$.
Since a knot in $\P$ with a lens space surgery is fibered and supports the tight contact structure (by Theorem~\ref{thm:properties}(2)), the Giroux Correspondence says that it may be obtained from $J$ by a sequence of plumbings and de-plumbings of positive Hopf bands.

\begin{question}
 If a knot in $\P$ has a lens space surgery, then can it be obtained from $J$ by a sequence of plumbings of Hopf bands?
\end{question}

We have yet to systematically check this question for the knots produced here or for the duals to the Tange and Hedden knots. 
As noted in Remark~\ref{rem:cable}, the knot $K_1$ is the $(2,3)$--cable of $J$. One may use this to show that $K_1$ is indeed obtained by plumbing two Hopf bands onto $J$.


\section{Notation and conventions}\label{sec:notation}
Let $K$ be a knot in an oriented $3$--manifold $M$.  Choose an orientation on $K$ and let $\mu$ be a meridian of $K$ in the torus $\bdry \nbhd(K)$ that positively links $K$.   Let $\lambda$ be an oriented curve in $\bdry \nbhd(K)$ that is isotopic to $K$ in $\nbhd(K)$; if $K$ is null-homologous in $M$, we choose $\lambda$ so that it is null-homologous in $M-\nbhd(K)$.   If $\gamma$ is an essential simple closed curve in $\bdry \nbhd(K)$, then when it is oriented $[\gamma] = p[\mu]+q[\lambda] \in H_1(\bdry \nbhd(K))$ where $p,q$ are coprime integers; changing the orientation of $\gamma$ changes the signs of both $p$ and $q$.   We refer to both the unoriented isotopy class of $\gamma$ in $\bdry \nbhd(K)$ and the  number $p/q \in \Q \cup \{\infty\}$ as a {\em slope} and (when $\lambda$ is null-homologous) say it is {\em positive} if $0<p/q<\infty$.   If the slope is integral, we also say it is a {\em longitude} or a {\em framing} of $K$.

The result of {\em $\gamma$--Dehn surgery} on the knot $K \subset M$ is the manifold  $M_K(\gamma)$
obtained by attaching $S^1\times D^2$ to $M-\nbhd(K)$ along the torus $\bdry \nbhd(K)$ so that ${\rm pt} \times \bdry D^2$ is identified with the slope $\gamma$.    The image of the curve $S^1 \times {\rm pt}$ in $M_K(\gamma)$ is a knot called the {\em surgery dual} to $K$ and denoted $K^*$.  Observe that $\gamma$ is the meridian of $K^*$ and $\mu$--surgery on $K^*$ returns $M$ with surgery dual $K=(K^*)^*$.

If $M$ is the double branched cover of $S^3$ over a link $L$, an arc $\kappa$ such that $\kappa \cap L = \bdry \kappa$  lifts to a knot $K$ in the cover $M$.  Via the Montesinos trick \cite{montesinostrick}, each integral surgery on $K$ corresponds to a {\em banding} along $\kappa$.  With $I=[-1,1]$, if $R=I \times I$ is a disk embedded in $S^3$ such that $R \cap L = \bdry I \times I$ and $\{0\} \times I = \kappa$, then the link $L' = (L - \bdry I \times I) \cup I \times \bdry I)$ is the result of a banding along $\kappa$.  The arc $\kappa^* = I \times \{0\}$ is the dual arc of the banding.  In the double branched cover over $L'$, $\kappa^*$ lifts to the dual knot $K^*$ of the corresponding integral surgery on $K$.

\smallskip
The lens space $L(p,q)$ is defined to be the manifold that results from $-p/q$--surgery on the unknot in $S^3$.  The two-bridge link $B(p,q)$ is the link in $S^3$ whose double branched cover is $L(p,q)$ \cite{HR}.  Using the continued fraction $[x_1, x_2, \dots, x_n] = x_1 -1/(x_2-1/(\dots -1/x_n))$ to express $-p/q$ describes the two-bridge link $B(p,q)$ geometrically in plat presentation as in the lower left of Figure~\ref{fig:P-235to2bridge}.

\smallskip
A non-trivial knot in a lens space $L(p,q)$ is {\em simple} (or {\em grid number one}) if, in the standard genus one Heegaard diagram of the lens space, it may be represented by two of the $p$ intersection points.  The simple knot is then the union of the arcs connecting those points in the two meridional disks whose boundaries are described by the diagram.  Equivalently, a simple knot (including the trivial knot) is a knot represented by a doubly pointed genus $1$ Heegaard diagram of $L(p,q)$ that has $p$ intersection points. There is one for each homology class $k \in H_1(L(p,q))$ in $L(p,q)$, denoted $K(p,q,k)$. See for example \cite{rasmussen,hedden,bgh}.

\section{Proofs}

\begin{proof}[Proof of Theorem~\ref{thm:main}]
Figure~\ref{fig:P-235to2bridge} exhibits arcs $\kappa_n$ on the pretzel knot $P(-2,3,5)$ that have a banding to a two-bridge link.
By passing to the double branched covers (and using the Montesinos Trick), this describes knots $K_n$ in $\P$ that have an integral surgery to a lens space.

Since $K_n$ lies in a genus $2$ Heegaard surface, its tunnel number is at most $2$.
If $K_n$ were to have tunnel number one, then every surgery on $K_n$ would have Heegaard genus at most $2$.  However we will see that surgery along the framing induced by the Heegaard surface produces a toroidal manifold which, according to Kobayashi's classification \cite{kobayashi}, does not have Heegaard genus $2$. (The same scheme was recently employed in \cite{EMJMM} to demonstrate a family of strongly invertible knots in $S^3$ with a Seifert fibered surgery and tunnel number $2$.)  

Kobayashi shows that if $M = M^\alpha \cup_T M^\beta$ is a genus $2$ manifold decomposed along a torus $T$ into two atoroidal manifolds, then one of $M^\alpha$ or $M^\beta$ admits a Seifert fibering over the disk with  $2$ or $3$ exceptional fibers or over the Mobius band with up to $2$ exceptional fibers  such that filling the other along the slope induced by a regular fiber in $T$ produces a lens space.  (This is more general than Kobayashi's classification but suitable for our needs.)   The slope of a regular fiber may be identified since filling one of these Seifert fibered spaces produces a reducible manifold if and only if the filling is done along the slope of a regular fiber.  

Figure~\ref{fig:toroidal}(Left) shows that the result of a $0$--framed banding along $\kappa_n$ has a sphere $S$ dividing it into two $2$--string tangles.  Up to homeomorphism, these two tangles are the tangle sums $T^\alpha_n = \tfrac{1}{n+2} + \tfrac{1}{3}$ and $T^\beta_n = \tfrac{1}{1-n} + \tfrac{1}{-2}$.  Their corresponding double branched covers, $M^\alpha_n$ and $M^\beta_n$, are in general each Seifert fiber spaces over the disk with exactly two exceptional (and non-degenerate) fibers:  $M^\alpha_n$ has type $D^2(|n+2|,3)$ and $M^\beta_n$ has type $D^2(|1-n|,2)$.   This fails for $M^\alpha_n$ when $n=-1,-2,-3$ and for $M^\beta_n$ when $n=0,1,2$; in each, the middle value yields a degenerate Seifert fibration while the other values yield solid tori.  In these cases the torus $T$ that is the double branched cover of $S$ is compressible and so Kobayashi's classification does not apply; moreover, one may observe that $K_n$ has tunnel number one in these cases.  Hence we assume $n \not \in \{-3,-2,-1,0,1,2\}$.

When $n=-1$ or $3$, $M^\beta_n$ has type $D^2(2,2)$ and thus is a twisted $I$--bundle over the Klein bottle.  Therefore it has an alternative Seifert fibration over the Mobius band with no exceptional fibers.  We already assume $n \neq-1$.  We shall assume $n\neq3$ as well.

\begin{figure}
\centering
\includegraphics[width=5.5in]{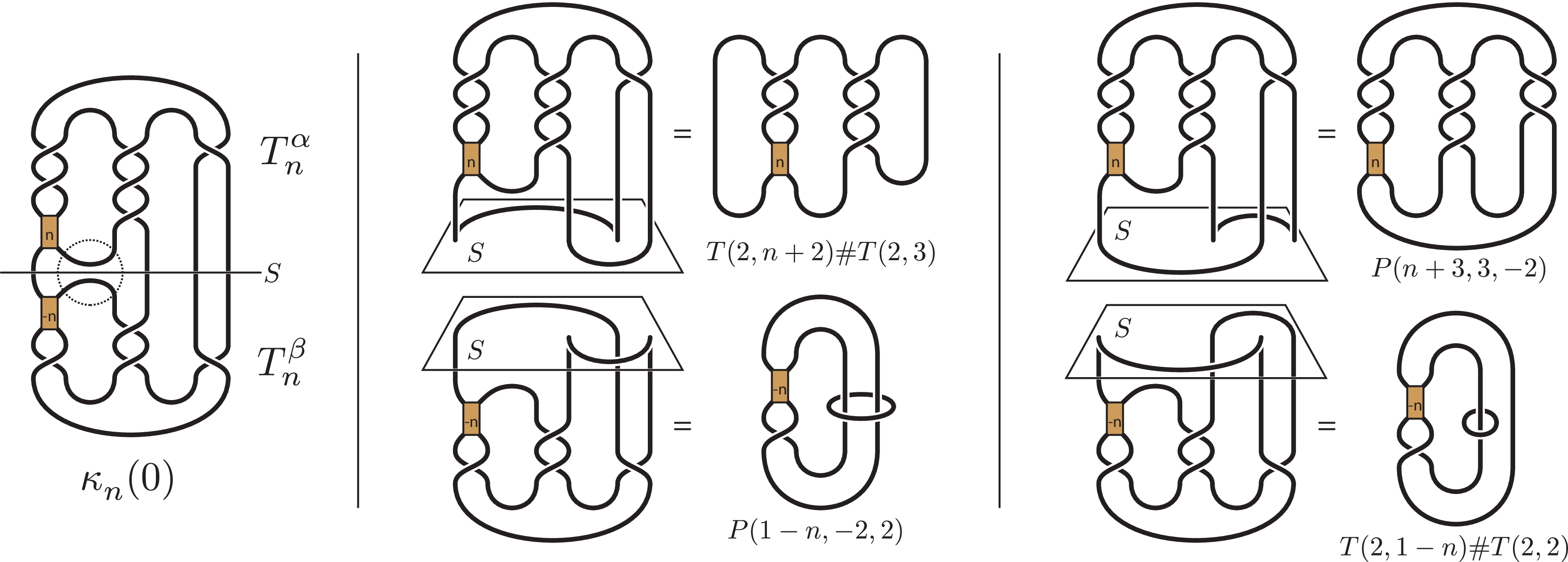}
\caption{(Left) The bridge sphere $S$ for $P(-2,3,5)$ splits $\kappa_n(0)$ into two tangles, $T^\alpha_n$ above and $T^\beta_n$ below. (Middle) Filling $T^\alpha_n$ with a rational tangle to get $T(2,n+2)\#T(2,3)$ makes $T^\beta_n$ into $P(1-n,-2,2)$.  (Right)  Filling $T^\beta_n$ with a rational tangle to get $T(2,1-n)\#T(2,2)$ makes $T^\alpha_n$ into $P(n+3,3,-2)$.     }
\label{fig:toroidal}
\end{figure}

Figure~\ref{fig:toroidal}(Middle) shows the results of filling $T^\alpha_n$ and $T^\beta_n$ with the rational tangle $\rho^\alpha$ defined by a pair of arcs in $S$ so that  $T^\alpha_n (\rho^\alpha)$ is composite.   We then observe that the filling $T^\beta_n(\rho^\alpha)$ is the pretzel link $P(1-n,2,2)$, and this is a two-bridge link if and only if $n =0,2$.   We have already omitted these values of $n$.

Figure~\ref{fig:toroidal}(Right) shows the results of filling  $T^\alpha_n$ and $T^\beta_n$ with the rational tangle $\rho^\beta$ defined by a pair of arcs in $S$ so that  $T^\beta_n (\rho^\beta)$ is composite.  We then observe that the filling $T^\alpha_n(\rho^\beta)$ is the pretzel link $P(n+2,3,-2)$, and this is a two-bridge link if and only if $n = -1,-3$.  We have already omitted these values of $n$ too.

Therefore, to conclude that the knot $K_n$ has tunnel number $2$, it is sufficient to require that $n \not \in \{-3,-2,-1,0,1,2,3\}$.
\end{proof}

\begin{lemma}\label{lem:homologous}
The  knot $K_n^*$ in $L(3n^2+n+1,-3n+2)$ that is surgery dual to $K_n$ is homologous to the knot $T_n^*$ that is surgery dual to the $(3n+1,n)$--torus knot in $S^3$.
\end{lemma}

\begin{proof}
Starting from the lower left of Figure~\ref{fig:P-235to2bridge} where the dual arc $\kappa_n^*$ is a presented on a standard form of the two bridge link $B(3n^2+n+1,-3n+2)$, Figure~\ref{fig:homology} (Left) shows how two clasp changes (and isotopy) transforms $\kappa_n^*$ into an arc $\tau_n^*$ in a bridge sphere. Since $\tau_n^*$ lies in the bridge sphere, its lift to the double branched cover is a knot $T_n^*$ in the Heegaard torus of the lens space.   Lifting the transformation of Figure~\ref{fig:homology} (Left) to the double branched cover thus shows that $K_n^*$ and $T_n^*$ are related by two crossing changes.  (The clasp changes lift to crossing changes as indicated in Figure~\ref{fig:homology}(Right).)  Hence these knots are homotopic, and therefore homologous, in the lens space.

\begin{figure}
\centering
\includegraphics[width=5.5in]{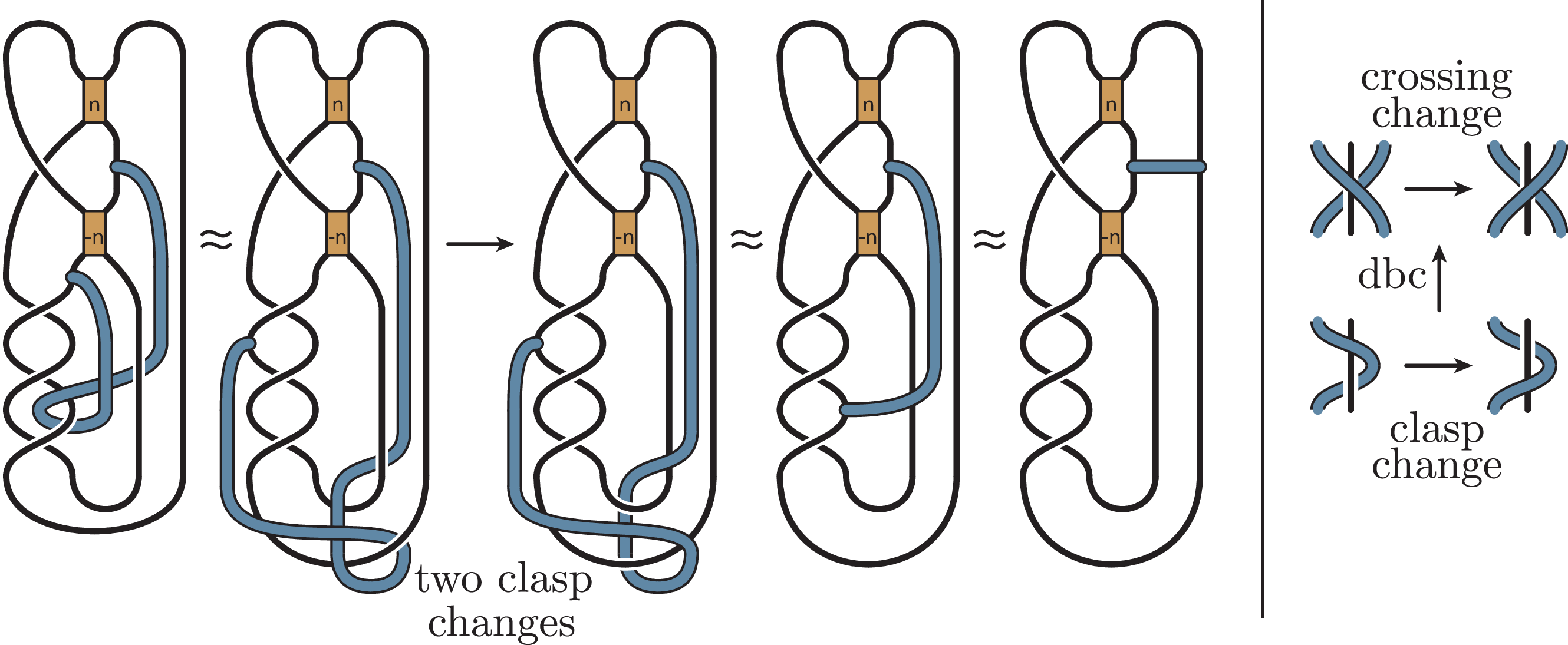}
\caption{(Left) An isotopy of the arc $\kappa_n^*$ on the two bridge link $B(3n^2+n+1,-3n+2)$ from the lower left of Figure~\ref{fig:P-235to2bridge} has two clasp changes performed.  The resulting arc $\tau_n^*$ is then isotoped to lie in a bridge sphere. (Right)  A clasp change of an arc around a segment of a branch locus lifts to a crossing change about a segment of the fixed set in the double branched cover.}
\label{fig:homology}
\end{figure}

\begin{figure}
\centering
\includegraphics[width=5.5in]{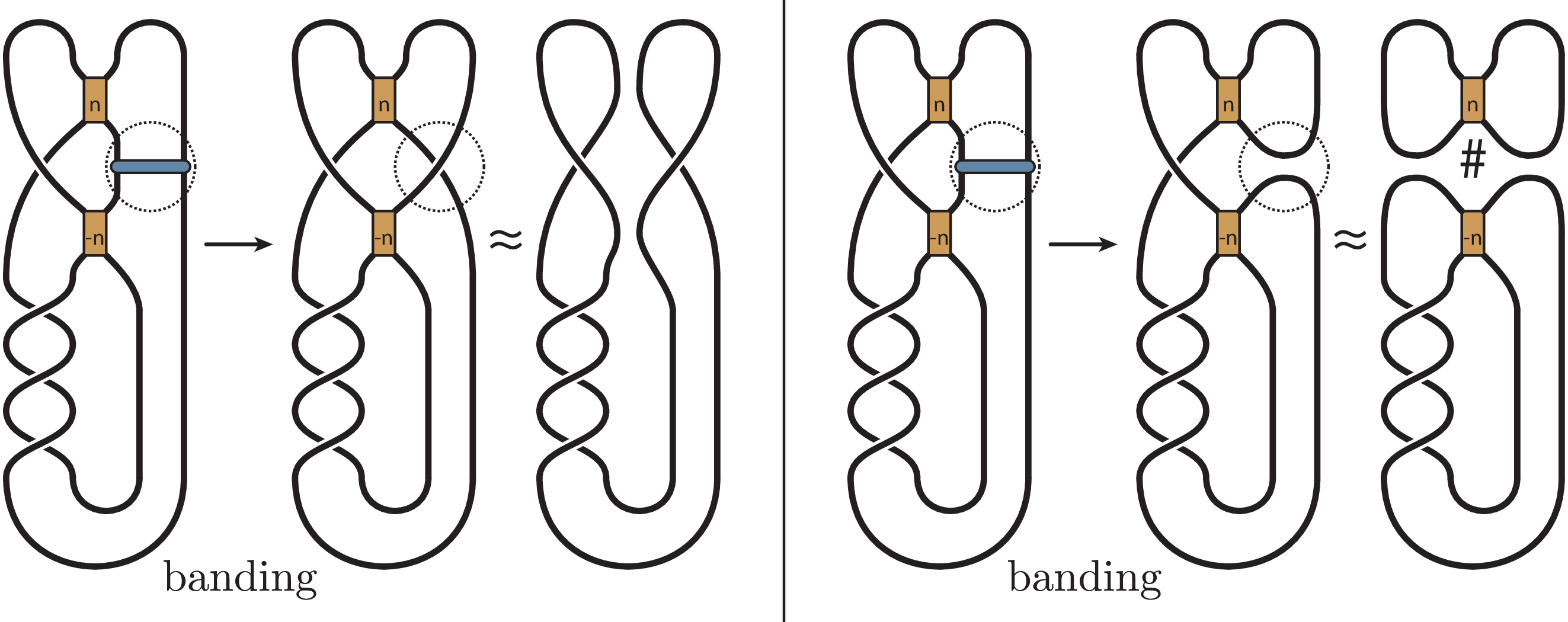}
\caption{(Left) A banding of $B(3n^2+n+1,-3n+2)$ along the arc $\tau_n^*$ produces the unknot. (Right)  A different banding of $B(3n^2+n+1,-3n+2)$ along the arc $\tau_n^*$ produces the connected sum of $B(n,-1)$ and $B(3n+1,3)$.}
\label{fig:torusknot}
\end{figure}

Finally,  Figure~\ref{fig:homology}(Left) shows a banding of $B(3n^2+n+1,-3n+2)$ along the arc $\tau_n^*$ to the unknot.  Thus we may identify $T_n^*$ as surgery dual to a torus knot $T_n$ in $S^3$.   Because Figure~\ref{fig:homology}(Left) shows a banding of $B(3n^2+n+1,-3n+2)$ along the arc $\tau_n^*$ to the connected sum $B(n,-1)\#B(3n+1,3)$, the torus knot $T_n$ has a reducible surgery to $L(n,-1)\#L(3n+1,3) \cong L(-n,3n+1) \# L(3n+1,-n)$ which is the mirror of $L(n,3n+1)\#L(3n+1,n)$.  Thus $T_n$ must be the $(3n+1,n)$--torus knot (for example see \cite{moser}).
\end{proof}

\begin{proof}[Proof of Theorem~\ref{thm:properties}]
(1) Theorem~\ref{thm:main} shows that either $\pm p$--surgery on $K_n$ is a lens space.   Tange shows that any integral lens space surgery on a knot in $\P$ must be positive \cite[Theorem 1.1]{tange-nonexistence}.

(2) In \cite[Theorem 3.1]{tange-nonexistence}, the paragraph following, and its proof, Tange further shows that if an L-space homology sphere $Y$ contains a knot $K$ with irreducible exterior for which a positive integer surgery produces an L-space, then $K$ is a fibered knot supporting a tight contact structure.  (Tange notes that the fiberedness of $K$ follows from the proof of \cite[Theorem 1.2]{OzSz-lens} and \cite{ni-fiberedknots}. Tange then demonstrates that the Heegaard Floer contact invariant of the contact structure supported by $K$ is non-zero.)

(5)  
If $T_n$ is  the $(3n+1,n)$--torus knot in $S^3$, then $p$--surgery on $T_n$ produces the lens space $L(p,q)$.   By Lemma~\ref{lem:homologous}, the surgery dual knot $T_n^*$ is homologous to the knot $K_n^*$ that is surgery dual to $K_n$.

(3) \& (6)
 Since the dual knots $K_n^*$ and $T_n^*$ are homologous, Theorem~\ref{thm:largesurgery} implies $2g(K^n) = p-1$ and the stated relationship of the Alexander polynomials of $K_n$ and $T_n$.

(4) Knowing that $g(K_n)=(p+1)/2$ implies $\Rk\HFK(L(p,q), K_n^*) = p+2$ \cite[Theorem 1.4]{hedden} (cf.\  \cite[Proposition 4.5]{rasmussen}).
\end{proof}

\begin{proof}[Proof of Theorem~\ref{thm:largesurgery}]
Let $K^*$ be the surgery dual to $K$ in $L(p,q)$ and (for some choice of orientations) let $J^*$ be the simple knot in $L(p,q)$ homologous to $K^*$.  Then $J^*$ has an integral surgery to an L-space homology sphere $Y_J$, \cite[Theorem 2]{rasmussen}.  

Since {\em positive} $p$--surgery on $K$ gives $L(p,q)$, then the self-linking number of $K^*$ is $-1/p \pmod{1}$, see \cite[Section 2]{rasmussen}.  This is demonstrated in Figure~\ref{fig:slopes}:  If $\Sigma$ is an oriented Seifert surface for $K$ giving the oriented longitude $\lambda = \bdry \Sigma$, then the positively linking meridian $\mu$ of $K$ is oriented as shown.  In order for $\lambda$ to run positively along $K^*$, we need $\mu^* \cdot \lambda >0$. Thus we must orient $\mu^*$ so that $[\mu^*] = [\lambda]+p[\mu]$.   This forces $\mu$ to be an anti-parallel longitude of $K^*$.  Hence we may use $-\mu$ to calculate the self-linking number of $K^*$ as $(-\mu \cdot \lambda)/p  \pmod{1}$.

With this set-up, $\mu$--surgery on $K^*$ may be regarded as $-1$--surgery on $K^*$.  This means $Y = K^*_{-1}$ in the notation of \cite{rasmussen}, though not explicitly stated.   Similarly we also have $Y_J = J^*_{-1}$.

\begin{figure}
\centering
\includegraphics[width=3.5in]{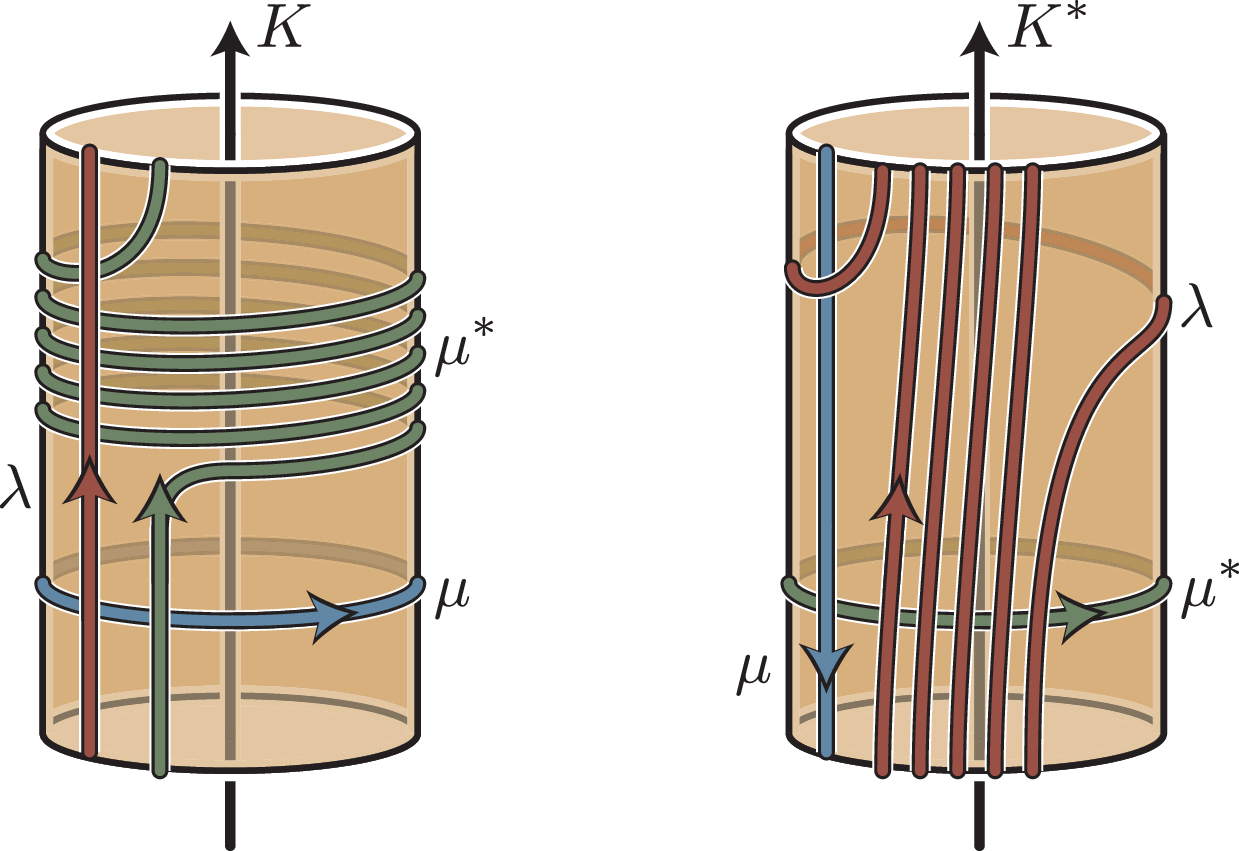}
\caption{(Left) A regular neighborhood of an oriented knot $K$ in a homology sphere $Y$ has meridian $\mu$ linking $K$ positively and longitude $\lambda$ oriented as the boundary of a Seifert surface.  The  surgery curve $\mu^*$ of positive slope $5/1$ shown here is oriented as $[\mu^*] = [\lambda] + 5 [\mu]$ in $H_1(\bdry \nbhd(K))$. (Right) Performing $\mu^*$--surgery on $K$ produces the rational homology sphere $Z$ and surgery dual knot $K^*$. Orienting $K^*$ to be linked positively by its meridian $\mu^*$, the boundary of the (rational) Seifert surface $\lambda$ is homologous to $5[K^*]$ in $\nbhd(K^*)$ while the longitude $\mu$ is homologous to $-[K^*]$.}
\label{fig:slopes}
\end{figure}

If $2g(K)-1 \neq p$, then the knots $J^*$ and $K^*$ have isomorphic knot Floer homology \cite[Theorem 2]{rasmussen} (cf.\ \cite[Theorem 1.4]{hedden}). 
Having isomorphic (and simple) knot Floer homology would imply that the L-space homology spheres $Y_J$ and $Y$ have the same $d$--invariants.  Hence $d(Y_J)=2$ too.  But now $Y_J$ cannot be $S^3$ since $d(S^3)=0$.

If $2g(K)-1=p$ then ${\rm width}\ \HFK(L(p,q),K^*) = 2p$ \cite{ni-linkfloer}, cf.\ \cite[Theorem 4.3]{rasmussen}. Thus $d(Y)=d(Y_J)+2$ by \cite[Proposition 5.4]{rasmussen} and so $d(Y_J)=0$. Greene's solution to the Lens Space Realization Problem \cite{greene} proceeds by first identifying the pairs (lens space, homology class)  that contain knots for which some integral surgery produces an L-space homology sphere with $d=0$, and then observing that each of these pairs contains the surgery dual to a Berge knot.  Since the lens space surgery duals to Berge knots are simple knots, we have that in fact $J^*$ is surgery dual to a Berge knot $B$ and $Y_J=S^3$.

\smallskip
For the statement about Alexander polynomials, 
let us first summarize work of Tange on Alexander polynomials of knots in L-space homology spheres for which a positive integral surgery yields a lens space \cite{tange-ozszcorrectionterm}.
The symmetrized Alexander polynomial of a knot $K$ in an L-space homology sphere with a non-trivial L-space surgery can be expressed as 
\[\Delta_K(t) = \sum_{i \in \Z} a_i(K) t^i = (-1)^k + \sum_{j=1}^k(-1)^{k-j}(t^{n_j}+t^{-n_j})  \in \Z[t^{\pm1}] \]
following the arguments of \cite{OzSz-lens}.   For a given positive integer $p$, pass to the quotient $\Z[t^{\pm1}]/(t^p-1)$ to obtain a polynomial 
$\tilde{\Delta}_K(t) = \sum_{i \in \Z/p\Z} \tilde{a}_i(K) t^i$ with coefficients $\tilde{a}_i(K) = \sum_{j = i \mod{p}} a_j(K)$.  Assuming that $p$--surgery gives a lens space $Z$, then $2g(K)-1 \leq p$ \cite{KMOS}. 
Note that $K$ is fibered%
\footnote{Since lens spaces are irreducible, a knot in a homology sphere with a lens space surgery must have irreducible exterior.  Thus $K$ has irreducible exterior and so it is fibered by \cite[Theorem 1.2]{OzSz-lens} and \cite{ni-fiberedknots}; cf.\ Theorem~\ref{thm:properties}(2).}
and hence the degree of $\Delta_K(t)$ equals $2g(K)$.
Because of this and that the coefficients of $\Delta_K(t)$ are $a_i(K) = 0$ or $\pm1$ (and the non-zero ones alternate in sign), it follows that  $\tilde{a}_i(K) = 0$, $\pm1$, or $2$ where  $\tilde{a}_i(K)=2$ implies that both $p$ is even and $i=p/2$ \cite[Corollary 2]{tange-ozszcorrectionterm}.  Furthermore, one may determine that the coefficients $\tilde{a}_i(K)$ and the polynomial $\tilde{\Delta}_K(t)$ only depend on the lens space $Z$ and the homology class of the surgery dual knot $K^*$ \cite{brody} (cf.\ \cite[\S5]{rasmussen}).  
In particular, if $p=2g(K)-1$ then 
\[\Delta_K(t) = \tilde{\Delta}_K(t) -(t^{(p-1)/2} +  t^{-(p-1)/2})+(t^{(p+1)/2} +  t^{-(p+1)/2})\]
where the indices for the coefficients of $\tilde{\Delta}_K(t)$ are taken to be the representatives of $\Z/p\Z$ from $-(p-1)/2$ and $(p-1)/2$.
See also the discussion preceding \cite[Proposition 3.3]{tangeknots}.

Now in our present situation, since the Berge knot $B$ in $S^3$ has $p$--surgery to $L(p,q)$ in which the surgery dual is the simple knot $J^*$, we therefore have $2g(B)<p$ because $p$ is odd \cite[Theorem 3.1]{hedden}\footnote{See also \cite[Theorem 1.4]{greene}.}.
  Hence $\tilde{a}_i(B) = a_i(B)$ for all $i \in \Z/p \Z$ and  $\Delta_{B}(t) = \tilde{\Delta}_{B}(t)$.
Because the surgery dual $K^*$  is homologous to the simple knot $B^*$ in $L(p,q)$,  $\tilde{\Delta}_{K}(t) = \tilde{\Delta}_{B}(t)$.   Therefore we have  
\[\Delta_{K}(t) = \Delta_{B}(t)  -(t^{(p-1)/2} +  t^{-(p-1)/2})+(t^{(p+1)/2} +  t^{-(p+1)/2})\]
as claimed.
\end{proof}

\begin{lemma}\label{lem:gofk}
There is a unique genus one fibered knot $J$ in $\P$.
As an open book, it supports the tight contact structure on $\P$.
Furthermore, it is surgery dual to the negative trefoil knot in $S^3$.
\end{lemma}
\begin{proof}

  This can be seen in the spirit of \cite{baker-cgofkils} which relates genus one fibered knot to axes of closed $3$--braids as follows: Noting that since $\P$ has a unique genus $2$ Heegaard splitting \cite{BoileauOtal},  $P(-2,3,5)$ is the only $3$--bridge link whose double branched cover gives $\P$. (In fact $\P$ is the double branched cover of no other link.)  Since $P(-2,3,5)$ is isotopic to the $(3,5)$--torus knot, its $3$--braid axis $A$ lifts to a genus one fibered knot $J \subset \P$.    By \cite[The Classification Theorem]{BirmanMenasco} and \cite[Lemma 3.8]{baker-cgofkils}, for instance, $A$ is the only $3$--braid axis for $P(-2,3,5)$ up to isotopy of unoriented links.  Hence $J$ is the only genus one fibered knot up to homeomorphism in $\P$.  Since the mapping class group of $\P$ is trivial \cite{BoileauOtal}, $J$ is the only genus one fibered knot in $\P$ up to isotopy.  

In the double branched cover, a braid presentation of the branch locus lifts to a Dehn twist presentation of the monodromy of the lift of the axis.
Since the axis $A$ presents $P(-2,3,5)$ as a positive braid (indeed, as the $(3,5)$-torus knot),  we obtain a presentation of the monodomy of $J$ as a product of positive Dehn twists.  Hence it supports the tight contact structure on $\P$ \cite{giroux}.

One may view $J$ as surgery dual to the negative trefoil in several ways.  Taking the route through branched covers, this is demonstrated in \cite[Proof 2]{baker-almostsimple}, though for the mirrored situation.
\end{proof}

\section{Hedden's almost-simple knots.}\label{sec:heddensknots}
In each lens space $L(p,q)$ with coprime $p>q>0$, Hedden diagramattically describes two unoriented $1$--bridge knots $T_L$ and $T_R$ ({\em the Hedden knots}) via the  doubly-pointed genus $1$ Heeggaard diagrams of $L(p,q)$ with $p+2$ intersection points \cite[Figure 3]{hedden}.  An alternative presentation of local pictures of these diagrams near the two marked points are shown in the top row of Figure~\ref{fig:almostsimplediagram}.   Since the {\em simple knots} in $L(p,q)$ are those that can be described via doubly-pointed genus $1$ Heeggaard diagrams of $L(p,q)$ with only $p$ intersection points, we like to regard the Hedden knots as {\em almost-simple}.  The bottom row of Figure~\ref{fig:almostsimplediagram} shows the same portion of the diagram for related simple knots.  These knots are notable because of the following theorem.

\begin{theorem}[{\cite[Theorem~3.1]{hedden}}]\label{thm:hedden3.1}
If $K^*$ is a $1$--bridge knot in $L(p,q)$ with an integral surgery to an L-space homology sphere, then either
\begin{itemize}
\item $p>2g-1$ and $K^*$ is a simple knot, or
\item $p=2g-1$ and $K^*$ is either $T_L$ or $T_R$.
\end{itemize}
where $g$ is the Seifert genus of the surgery dual knot.
\end{theorem}  
Though the statement of this theorem in \cite{hedden} is explicitly only for surgeries to $S^3$, Hedden notes that his proof applies for surgeries to any L-space homology sphere.

\begin{figure}
\centering
\includegraphics[width=3.5in]{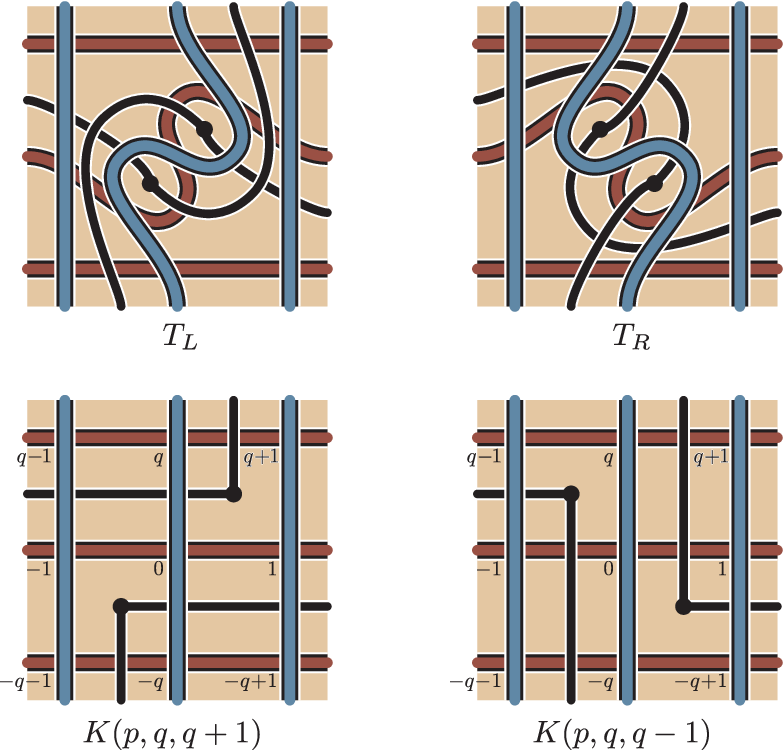}
\caption{Local portions of the doubly pointed diagrams for the almost-simple knots $T_L$ and $T_R$ and their associated simple knots $K(p,q,q+1)$ and $K(p,q,q-1)$; cf.\ \cite[Figure 1]{baker-almostsimple} and \cite[Figure 3]{hedden}.  In the botton two pictures, the intersection points of the red $\alpha$--curve and blue $\beta$--curve are numbered in order mod $p$ along the $\alpha$--curve.}
\label{fig:almostsimplediagram}
\end{figure}

Here we would like to clarify a few points about these almost-simple knots,  correct a couple of misstatements in \cite{baker-almostsimple}, and discuss the homology spheres that may be obtained by integral surgery on them following the arguments of \cite{baker-almostsimple}.  Afterwards, we summarize these results in Proposition~\ref{prop:almostsimple}.

\bigskip
The knots $T_L$ and $T_R$ are not homologous in general, contrary to what was stated in \cite{baker-almostsimple}.  One can observe this directly from the descriptions in Figure~\ref{fig:almostsimplediagram} (or \cite[Figure 3]{hedden}), noting that the magnitude of their algebraic intersection numbers with the blue $\beta$--curve differ by $2$.  One may further observe that the mirror of the knot $T_L$ in $L(p,q)$ is the knot $T_R$ in $L(p,-q)$. 

Furthermore, the proof in \cite{baker-almostsimple} that $T_L$ differs from (and thus is homologous to)  $K(p,q,q+1)$ by a crossing change also shows that $T_R$ differs from $K(p,q,q-1)$ by a crossing change.  Unfortunately, \cite[Figure~2]{baker-almostsimple} should be mirrored through the Heegaard torus so that the red $\alpha$--disks are below,  the blue $\beta$--disks are above, and the twist in $T_L$ is right-handed in order to be consistent with the orientation conventions.  In particular, the crossing change from $K(p,q,q+1)$ to $T_L$ is achieved by a $-1$--surgery on a loop $C_L$ in the Heegaard that encircles the two base-points in the diagram for $K(p,q,q+1)$ of bottom left Figure~\ref{fig:almostsimplediagram}.   The crossing change from $K(p,q,q-1)$ to $T_R$ may be similarly achieved by  a  $+1$--surgery on a loop $C_R$.

A knot homologous to $K(p,q,k)$ has an integral surgery to a homology sphere if and only if its self linking is $\pm1/p \pmod{1}$, or equivalently if and only if $k^2 = \pm q \pmod{p}$ \cite[Lemma 2.6]{rasmussen}, cf.\ \cite{rons}.  Here the choice of sign of $\pm$ is consistent and agrees with whether this integral surgery is a $\pm1$--surgery; cf.\  Figure~\ref{fig:slopes} and the discussion in the proof of Theorem~\ref{thm:largesurgery}.

If $T_L$ in $L(p,q)$ is dual to positive $p$--surgery on a knot $K$ in a homology sphere $Y$, then so is $K(p,q,q+1)$ and hence $(q+1)^2 = -q \pmod{p}$.  Making the substitution $-k = q+1$ gives the equation $k^2+k+1 = 0 \pmod{p}$. Thus $K(p,q,q+1)$ is dual to a type VII Berge knot $B$, one that lies in the fiber of a trefoil knot; see \cite[Section 6.2]{rasmussen} and \cite[Section 1.2]{greene}.  Since its fiber positively frames $B$, this trefoil knot is the {\em negative} trefoil (not the positive trefoil
as stated in \cite{baker-almostsimple}), and it can be identified with $C_L$ under the $S^3$ surgery on $K(p,q,q+1)$. Performing $-1$--surgery on the negative trefoil $C_L$ in $S^3$ produces $\P$ and takes $B \subset S^3$ to $K \subset \P$.  Since the linking of $B$ and $C_L$ is $0$, the positive $p$--surgery on $B$ becomes the positive $p$--surgery on $K$.   Consequently, if $-1$--surgery on $T_L$ is a homology sphere, it is $\P$. 

Similarly, if $T_R$ in $L(p,q)$ is dual to positive $p$--surgery on a knot $K'$ in a homology sphere, then $(q-1)^2 = -q \pmod{p}$.  Making the substitution $k = q-1$ gives the equation $k^2-k-1 = 0 \pmod{p}$.  Thus $K(p,q,q-1)$ is dual to a type VIII Berge knot $B'$, one that lies in the fiber of the figure eight knot.  This figure eight knot may be identified with $C_R$.  Performing $+1$--surgery on $C_R$ in $S^3$ produces the Brieskorn sphere $\Sigma(2,3,7)$ taking the knot $B'$ to the knot $K'$.  Consequently, if $-1$--surgery on $T_R$ is a homology sphere, it is $\Sigma(2,3,7)$.

This coincides with the difference between the $\tau$--invariants of these knots as noted by Rasmussen in the last two paragraphs of \cite[Section 5]{rasmussen}:  $\tau(T_L, \mathfrak{s}_0) =-1$ and so $\tau(T_R, \mathfrak{s}_0) =+1$.
Rasmussen further shows that if integral surgery on $T_L$ or $T_R$ produces a homology sphere, then it is an L-space homology sphere if and only if the surgery is a $-1$--surgery on $T_L$ or a $+1$--surgery on $T_R$ \cite[Proposition 4.5]{rasmussen}. 

In summary, the above discussion shows:
\begin{proposition}\label{prop:almostsimple} \
\begin{enumerate}
\item In $L(p,q)$, $T_L$ is homologous to the simple knot $K(p,q,q+1)$ and $T_R$ is homologous to the simple knot $K(p,q,q-1)$.

\item The mirror of $(L(p,q),T_L)$ is $(L(p,-q),T_R)$.

\item
If $-1$--surgery on $T_L$ is a homology sphere, then it is $\P = \Sigma(2,3,5)$ and $K(p,q,q+1)$ is positive surgery dual to a type VII Berge knot.

\item
If $-1$--surgery on $T_R$ is a homology sphere, then it is $\Sigma(2,3,7)$ and 
$K(p,q,q-1)$ is positive surgery dual to a type VIII Berge knot.
\end{enumerate}
\end{proposition}

\section{Acknowledgments}
This article has benefited from discussions with Joshua Greene and Tye Lidman. Parts of this work and writing were done during visits to the University of Texas and Boston College in the winter of 2015; we thank them for their hospitality.

KLB was partially supported by a grant from the Simons Foundation (\#209184 to Kenneth L.\ Baker).

\medskip
\myrule{open diamond}{open diamond}
\medskip

\appendix\def\leftmark{Neil R.\ Hoffman}
\def\rightmark{The Poincar\'e homology sphere, lens space surgeries (appendix)}
\section{Symmetries of knots with lens space and Poincar\'e homology sphere surgeries} 
\begin{center}
{\sc Neil R.\ Hoffman}
\end{center}

The aim of this appendix is to compute the symmetry group of a hyperbolic knot in a lens space with a surgery to the Poincar\'{e} homology sphere $\P$.
Specifically, it will provide a proof of the following theorem:

\begin{theorem}\label{thm:symm}
Let $X$ be a one-cusped hyperbolic manifold admitting both a Poincar\'{e} homology sphere filling and a lens space filling. Then the symmetry group of $X$ is trivial or generated by a single strong involution.
\end{theorem}

Our method blends the classification of spherical orbifolds with results involving exceptional Dehn fillings. We refer the reader to \cite{BBCW2012} for the most closely related paper on this topic and Thurston's notes \cite[Chapter 13]{thurston1979geometry} for a more general introduction to orbifolds. 
In preparation, we first extend the notation of section~\ref{sec:notation} to address orbifolds and orbifold fillings and then discuss symmetries and quotients of hyperbolic manifolds.
Thereafter, the argument in this appendix will gradually ``whittle-down'' the symmetry group of the exterior $X$ of a hyperbolic knot $K$ in a lens space with a Dehn surgery to $\P$. First, we give an argument which eliminates orientation reversing symmetries so that $Sym(X) = Sym^+(X)$.  Then we proceed to analyze a subgroup $Z(X)$ of  $Sym^+(X)$, showing (a) that it has index at most $2$, (b) if the index is $2$ then $Sym(X)$ contains a strong involution, and ultimately (c) that $Z(X)$ is trivial.  This is pulled together in the proof of Theorem~\ref{thm:symm} at the end of this appendix.

\subsection{Background on orbifolds and symmetries.}\label{sec:Abackground}
Let $Q$ be a $3$--dimensional manifold or orbifold with {\em singular set} $\Sigma(Q)$, the set of points fixed by some non-trivial element of $\pi_1^{orb}(Q)$.  
(So $Q$ is a manifold if $\Sigma(Q)=\emptyset$.) If $K$ is a knot in $Q$, then its complement is $Q-K$ while its exterior is $X=Q-\nbhd(K)$.  We further assume that  $K$ is either  a component of $\Sigma(Q)$ or disjoint from $\Sigma(Q)$.

The torus $\bdry \nbhd(K)$ represents the \emph{cusp} of $Q-K$ corresponding to $K$.  In such a torus we consider essential closed curves  up to free homotopy.  A \emph{primitive curve} is homotopic to an essential simple closed curve, and hence is a slope.  Any non-primitive essential curve is a multiple of a primitive curve.  
Given two closed curves $\alpha$ and $\beta$ in the torus we say the \emph{distance} between $\alpha$ and $\beta$, $\Delta(\alpha, \beta)$ is the minimal (unoriented) geometric intersection number between the two curves.
 If $\alpha$ is a curve in $\bdry \nbhd(K)$ that is an $n$--fold multiple of a slope $\gamma$, then the result of $\alpha$--Dehn surgery on $K \subset Q$ is the orbifold  $X(\alpha)$ with underlying manifold $|X(\alpha)|$ obtained as the $\gamma$--Dehn surgery on $K \subset |Q|$  and a new singular set  $\Sigma(X(\alpha))$ consisting of $\Sigma(X)$ and the core of the attached $S^1 \times D^2$ with order $n$ (unless $n=1$ in which case the singular set remains only $\Sigma(X)$).  The orbifold attached to $X$ is a {\em orbi-solid torus} of order $n$.  An orbi-solid torus of order $1$ is just a solid torus.
 
 Following \cite[\S3]{BBCW2012}, an {\em orbi-lens space} is the quotient of $S^3$ by a finite cyclic subgroup of $SO(4)$ and is consequentially a union of two orbi-solid tori with coprime orders along their common boundary.  Its singular set is the union of the cores of these orbi-solid tori that have order at least $2$.   

A knot $K$ in a closed $3$--manifold $M$ is {\em hyperbolic} if its complement $M-K$ is a hyperbolic $3$--manifold with a single torus cusp.  We also say the exterior $X = M-\nbhd(K)$ of such a knot is {\em hyperbolic} if the interior of $X$ is hyperbolic.  For a hyperbolic knot $K$ with exterior $X$, the symmetry group $Sym(X)$ of homeomorphisms modulo isotopy is identified with the isometry group of the interior of $X$.  
The orientation preserving symmetry group is denoted $Sym^+(X)$.

Within $Sym^+(X)$ is the subgroup $Z(X)$, the maximal subgroup that acts freely on $\bdry X$.  In particular, $Z(X)$ is the maximal subgroup of $Sym^+(X)$ for which the quotient orbifold (or manifold) $X/Z(X)$ has a torus cusp.

Since $X$ is hyperbolic, the quotient $X/Sym^+(X)$ is an orientable cusped hyperbolic $3$--orbifold with $vol(X/Sym^+(X)) = \vol(X)/|Sym^+(X)|$.  Since  the minimal volume of an orientable cusped hyperbolic $3$--orbifold is positive (e.g.\ \cite{meyerhoff1986sphere}) and $X$ has finite volume, it follows that $Sym^+(X)$ is a finite group.

In a knot exterior $X$, the {\em homological longitude} is the slope $\lambda$ in $\bdry X$ so that $[\lambda]=0$ in $H_1(X; \Q)$.  This slope is unique up to sign as it is a generator of the kernel of the map $i_* \colon H_1(\bdry X; \Q) \to H_1(X; \Q)$ which has rank $1$ by the Half-Lives, Half-Dies Lemma.  Since it is unique, elements of $Sym(X)$ must preserve the slope $\lambda$ up to sign.
Elements of $Sym^+(X)$ of order $2$ that reverse the orientation of the homological longitude $\lambda$ of the manifold $X$ are called {\em strong involutions}.

If a group $G$ acts on a hyperbolic knot exterior  $X$, the cusp of the quotient orbifold $X/G$ records important information about the action of $G$ on the cusp of $X$. If the quotient $X/G$ has a torus cusp, then $G$ acts freely on  $\bdry X$. Furthermore, each element of $G$ restricts to a translation on $\bdry X$.  If $X/G$ has an $S^2(2,2,2,2)$ cusp,  then $G$ contains an order 2 element which fixes a discrete set of points on $\bdry X$.  In particular, if $X/Sym^+(X)$ has an $S^2(2,2,2,2)$ boundary, then  such an  element of $Sym^+(X)$ must preserve isotopy classes of slopes in $\bdry X$ but reverse their orientations.  Hence $Sym^+(X)$ contains a strong involution.  Furthermore, this means that there is a degree $2$ cover with a torus cusp, and hence $Z(X)$ has index $2$ in $Sym^+(X)$.

\subsection{Whittling down the symmetry group}

\begin{proposition}\label{prop:noReverse}
Let $X$ be the exterior of a knot in a lens space $L(p,q)$ with $p\geq 2$
 admitting a Dehn surgery to an integral homology sphere.
Then either $X$ is Seifert fibered or $Sym(X) = Sym^+(X)$.
\end{proposition}

%
%
\begin{proof}
Assume $X$ admits an orientation reversing symmetry $\tau$, and let 
$\lambda$ be the homological longitude of $X$. 
Then $\tau(\lambda)$ is isotopic to $\pm \lambda$ in $\bdry X$.
Since $\tau$ restricts to a homeomorphism on the boundary torus $\bdry X$, there must be a slope $\delta$ with $\Delta(\delta,\lambda)=1$ such that $\tau(\delta) = \pm \delta$.  Because this restriction of $\tau$ to $\bdry X$ must also be orientation reversing, $\tau$ reverses the orientation on exactly one of $\delta$ and $\lambda$.

Let $\mu$ be the slope in $\bdry X$ such that $X(\mu) = L(p,q)$.   
Since $X$ has a filling producing an integer homology sphere, $X$ must be a homology solid torus.  
Thus, $\mu= p \delta + q'\lambda$ in $H_1(\bdry X)$ for some integer $q'$ coprime to $p$ and $p\geq 2$.   Since $\tau(\mu) = \pm(p\delta-q'\lambda)$ and $q' \neq 0$, we have that $\Delta(\mu, \tau(\mu)) = |2pq'|>2$. 

 The symmetry $\tau$ implies  $X(\mu) \cong X(\tau(\mu))$ so that $\pi_1(X(\mu))$ and $\pi_1(X(\tau(\mu)))$ are both cyclic.   Thus, since $\Delta(\mu, \tau(\mu)) > 1$,  $X$ must be Seifert fibered by the Cyclic Surgery Theorem \cite{CGLS}. 
\end{proof}

We now further classify $Z(X)$.

\begin{lemma}\label{lem:coverandquotientnew}
Let $X$ be the exterior of a hyperbolic knot $K$ in a lens space $L(p,q)$ with $p\geq 2$ admitting a Dehn surgery to an integral homology sphere.
Then the following hold:
\begin{enumerate}
\item $Z(X)$ is cyclic.
\item $[Sym^+(X):Z(X)] \leq 2$
\item If  $[Sym^+(X):Z(X)] = 2$, then $X$ has a symmetry which is a strong involution.
\item In the quotient orbifold $\mathcal{Z} = X/Z(X)$, the meridian $\mu$ of $K$ maps to a primitive curve $\bar{\mu}$ in $\bdry\mathcal{Z}$.
\item  The Dehn filling $\mathcal{Z}(\bar{\mu})$ is an orbi-lens space in which the core $\bar{K}$ of the filling is disjoint from the singular set of $\mathcal{Z}(\bar{\mu})$.
\end{enumerate}
\end{lemma}

\begin{proof}
We first utilize the fact that $X$ is the exterior of a hyperbolic knot in an integral homology sphere to
show that the boundary of $X/Sym^+(X)$ is either a torus or $S^2(2,2,2,2)$.

Since $X$ admits an integer homology sphere filling, $X$ is a homology solid torus.  Furthermore, once oriented, the slope $\gamma$ in $\bdry X$ of this filling generates $H_1(X)$.  Thus together with a choice of orientation on the homological longitude $\lambda$ it forms a basis $(\gamma, \lambda)$ for $H_1(\bdry X)$.

Now consider the action of $Sym^+(X)$ on $\bdry X$.  Since the homological longitude is the unique slope in $\bdry X$ that is null homologous in $H_1(X;\Q)$, 
 any element of $Sym^+(X)$ must send $\lambda$ to $\pm \lambda$.  Since $Sym^+(X)$ is finite because $X$ is hyperbolic, consideration of the action on $H_1(\bdry X;\Q)$ shows that any element of $Sym^+(X)$ must send $\gamma$ to $\pm \gamma$. 
Then, because $Sym^+(X)$ is orientation preserving on $\bdry X$ too, any element of $Sym^+(X)$ must send the basis $(\gamma, \lambda)$ to $\pm(\gamma, \lambda)$.  Hence all slopes in $\bdry X$ are preserved by the action of $Sym^+(X)$. Consequently, the restriction to $\bdry X$ of a symmetry of $X$ either is trivial, has no fixed point (and hence is a translation), or  has order $2$.   
Therefore, by consideration of the Euclidean orbifolds which are orientable quotients of the torus, the boundary of $X/Sym^+(X)$ is either a torus or $S^2(2,2,2,2)$. 

\medskip
 
 We now further refine this analysis by incorporating the lens space filling.
Since its exterior $X$ is a homology solid torus, the knot $K$ in $L(p,q)$ represents a generator of $H_1(L(p,q),\Z)$.  Hence it lifts to a knot $\widetilde{K} \subset S^3$
 in the $p$--fold cyclic cover $S^3 \to L(p,q)$.
This cover restricts to a $p$--fold cyclic cover of knot exteriors $\X \to X$ in which the meridian $\mu \subset \bdry X$ of $K$ lifts to $p$ disjoint copies of the meridian of $\widetilde{K}$.  (Here $\X$ is the exterior of $\widetilde{K}$.)  The deck transformation group $H$ of this cover is therefore isomorphic to the cyclic group $\Z/p\Z$ and extends to an action on $S^3$ that is free on $\bdry \X$.   Hence $H$ is a subgroup of $Z(\X)$.

Consider the composition $\alpha \colon \X \to X/Sym^+(X)$ of the covers $\X \to X$ and $X \to X/Sym^+(X)$ with deck transformation groups $H$ and $Sym^+(X)$ respectively.  Let $G$ be the deck transformation group of $\alpha$.
If $\alpha$ is regular, then $G$ is a subgroup of $Sym^+(\X)$ and $H$ must be a normal subgroup of $G$ with $Sym^+(X) = G/H$.
 As the boundary of $X/Sym^+(X)$ is either a torus or $S^2(2,2,2,2)$, we claim that $\alpha$ is indeed a regular cover.
If this boundary is a torus, then this follows directly from \cite[Lemma 4]{Reid1991}.  However if this boundary is $S^2(2,2,2,2)$, then the argument of \cite[Lemma 4]{Reid1991} extends to also show that the covering $\alpha \colon \X \to X/Sym^+(X)$ is a regular cover as noted in the proof of \cite[Proposition 9.1]{neumann1992arithmetic}.  (Ultimately, this is because conjugation by an order $2$ element of the peripheral subgroup $\pi_1(S^2(2,2,2,2))$ inverts elements of the maximal abelian subgroup.  This maximal abelian subgroup contains the image of the peripheral subgroup of $\X$.)

Because $\X$ is the exterior of a hyperbolic knot in $S^3$, we know that $Sym^+(\X)$ must be cyclic or dihedral \cite[top of page 627]{BBCW2012} (see also  \cite[Proposition 9.1]{neumann1992arithmetic}). 
More specifically $Z(\X)$ is cyclic, $[Sym^+(\X):Z(\X)] \leq 2$, and $[Sym^+(\X):Z(\X)] = 2$ only if $Sym^+(\X)$ contains a strong involution.

Since $G$ is a subgroup of $Sym^+(\X)$, either it is  a subgroup of $Z(\X)$ and hence cyclic and acting freely on $\bdry \X$ itself, or it contains a strong involution.  In the former case, $Sym^+(X) = G/H$ is cyclic as it is the quotient of a cyclic group $G$ and  $X/Sym^+(X)$ has torus boundary so that $Sym^+(X)=Z(X)$.  In the latter case, $Sym^+(X) = G/H$ must also contain a strong involution since $H$ is a subgroup of $Z(\X)$.   Furthermore, because $H \cong \Z/p\Z$ is a non-trivial subgroup of $Z(\X)$ (since $p\geq 2$),  $G$ cannot be generated by a single strong involution.  Thus $G$ must be a dihedral group (of order $\geq 4$), and so $Sym^+(X)=G/H$ is a dihedral group too (though possibly of order $2$).  Since $X/Sym^+(X)$ has $S^2(2,2,2,2)$ boundary, $[Sym^+(X):Z(X)]=2$ as noted in the introduction to this appendix.  Since $Z(X)$ cannot contain a strong involution, is a cyclic group. Consequently, in both cases claims (1), (2), and (3) of the lemma hold.

\medskip

Finally we examine the quotient orbifold $\mathcal{Z}=X/Z(X)$ and its filling $\mathcal{Z}(\bar{\mu})$.
Since $\X$ is the exterior of the knot $\widetilde{K} \subset S^3$, $Sym^+(\X)$ naturally identifies with $Sym^+(S^3, \widetilde{K})$.   In particular, $Z(\X)$ acts freely on $\widetilde{K}$.   Therefore, if $\widetilde{\mu}$ is the lift of the slope $\mu$ to $\X$ and $\bar{\mu}$ is its image in the quotient orbifold $\mathcal{Z}=\X/Z(\X)={X}/Z({X})$,
 then ${\bar{\mu}}$ must be a primitive slope and (4) holds.
 
  Filling $\mathcal{Z}$ along the slope $\bar{\mu}$ produces 
 \[\mathcal{Z}(\bar{\mu}) = X(\mu)/Z(X)\] which we recognize as an orbi-lens 
  space since it is a cyclic quotient of $S^3 = \X(\widetilde{\mu})$ (see \cite[\S 3]{BBCW2012} for context).
  Because $\bar{\mu}$ is a primitive slope, the core of the solid torus in the filling $\mathcal{Z}(\bar{\mu})$  is a knot $\bar{K}$ that is disjoint from the singular set $\Sigma(\mathcal{Z}(\bar{\mu}))$ of the orbi-lens space.   In particular, $\mathcal{Z}$ is the exterior of a knot in an orbi-lens space that is disjoint from the singular set. Thus (5) holds, which completes the proof.
 \end{proof}

In light of the above lemma which employs a classification of cyclic quotients of $S^3$, it will also prove useful to have a list of manifolds and orbifolds that are cyclic quotients of $\P$. 
This is accomplished by Lemma~\ref{lem:dunbar} below which essentially follows from combining the classification of elliptic manifolds with the corresponding classification of elliptic orbifolds in Dunbar \cite{dunbar1988geometric}.

\begin{lemma}\label{lem:dunbar}
If the finite cyclic quotient of $\P$ by symmetries is an orbifold $Q$, then $Q$ is homeomorphic to one of the following:
\begin{enumerate}
\item a manifold $M$ with $\pi_1(M) \cong \pi_1(\P) \times \Z/n\Z$ for some integer $n$ coprime to $30$,
\item an orbifold that fibers over $S^2(2,3,5)$ where the singular set of $Q$ has $1$, $2$, or $3$ components, or
\item an orbifold with base space $S^3$ and singular set a knot or link as pictured in Figure \ref{fig:singSets}.
\end{enumerate}
Moreover, in the case that $Q$ is an orbifold, the components of its singular set are fixed locally by finite cyclic groups of relatively prime orders.
\end{lemma}

\begin{proof}
By the Orbifold Theorem (see \cite{boileau2003three} and \cite[Chapter 7]{cooperthree}), we know that that $Q$ is a spherical 3-orbifold.
 
Case 1 follows directly from the classification of elliptic manifolds (see for example \cite[Theorem 4.4.14]{thurston1997three}), while Cases 2 and 3 follow from a careful reading of  \cite{dunbar1988geometric}. In \cite{dunbar1988geometric}, Dunbar also classifies orbifolds covered by elliptic manifolds 
 where the singular sets are trivalent graphs. However, no orbifold of this type can be cyclically covered by a manifold because the group of isometries that fixes a vertex in the trivalent graph is never cyclic.  
\end{proof}

\begin{figure}[ht]
\centering
\includegraphics[width=5 in]{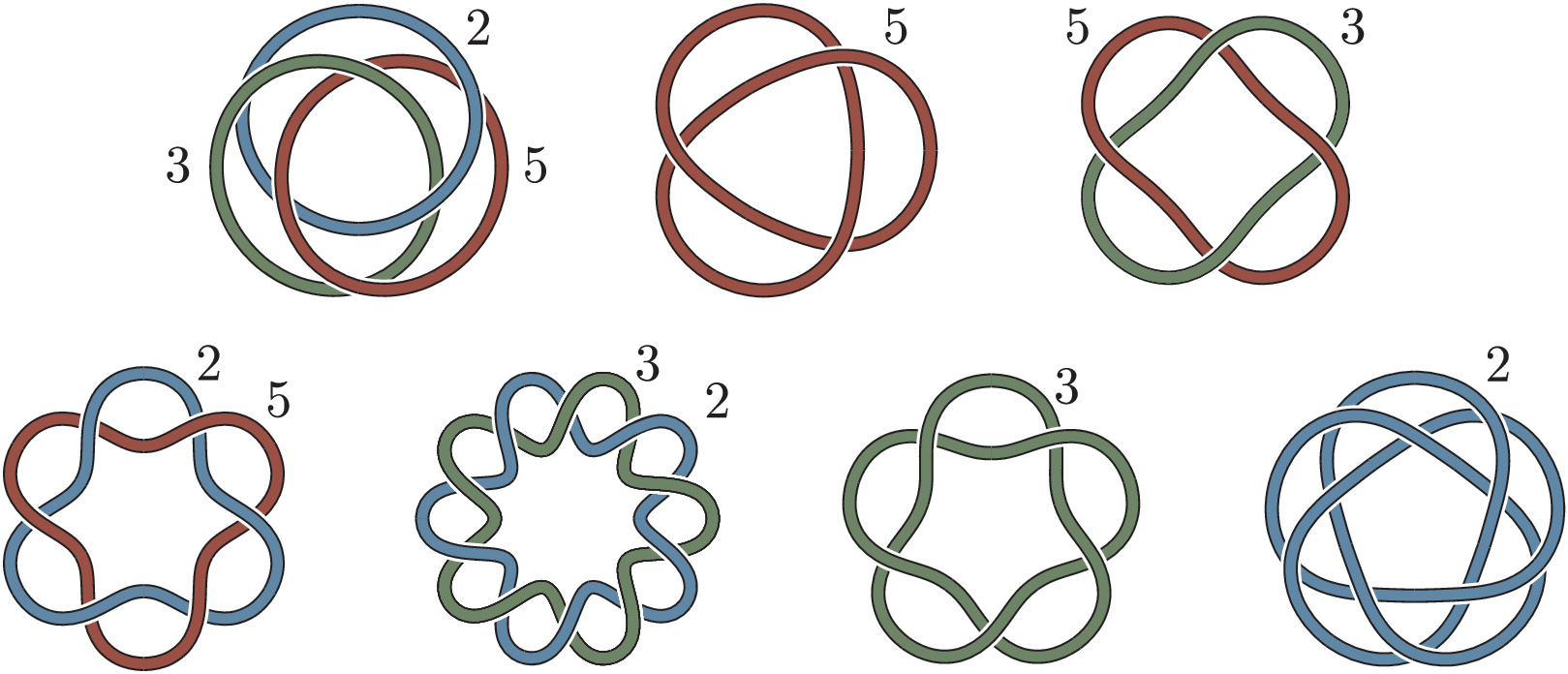}
\caption{ The possible singular sets of orbifolds cyclically covered by $\P$ with underlying space $S^3$. These are the relevant cases from a more general classification in \cite{dunbar1988geometric}. 
Following the notation of that paper, a link component labeled by $n$ indicates a cone angle of $2\pi/n$ along that component in the singular set of the corresponding orbifold.
}
\label{fig:singSets}
\end{figure}

A simple case analysis shows that drilling the singular sets of orbifolds from the previous lemma results in several interesting manifolds. 

\begin{lemma}\label{lem:drilling}
In Cases 2 and 3 of Lemma \ref{lem:dunbar}:
\begin{itemize}
\item If $|\Sigma(Q)|=1$, 
then the complement of  $\Sigma(Q)$ is a Seifert fibered space over the disk with two exceptional fibers.
\item  If $|\Sigma(Q)|=2$,
then the complement of $\Sigma(Q)$ is a Seifert fibered space over the annulus with one exceptional fiber.
\item If $|\Sigma(Q)|=3$, 
then the complement of $\Sigma(Q)$ is $F\times S^1$, where $F$ is a pair of pants. 
\end{itemize} \qed
\end{lemma}

\begin{lemma}\label{lem:noFreeAction}
Let $X$ be the exterior of a hyperbolic knot $K$ in a lens space $L(p,q)$ with $p\geq 2$ admitting a Dehn surgery to $\P$.
Then every non-trivial subgroup of $Z(X)$ acts non-freely on $X$.
\end{lemma}

\begin{proof}

Assume $Z_f \subset Z(X)$ is a non-trivial subgroup of elements that act freely on $X$.  Then the quotient $X/Z_f$ is a hyperbolic manifold with a torus boundary.  Since $Z_f$ is non-trivial, $|Z_f| \geq 2$.
Since $Z(X)$ is cyclic by Lemma~\ref{lem:coverandquotientnew}, the subgroup $Z_f$ must be cyclic.

Assume $\mu$ is the meridian of the knot $K$ in the lens space, and $\gamma$ is the slope of is $\P$ surgery.  
Let $\bar{\mu}$ and $\bar{\gamma}$ be the images of $\mu$ and $\gamma$ in the quotient $X/Z_f$.
Then the orbifold fillings of $X/Z_f$ are quotients of the Dehn fillings of $X$.  In particular,
$(X/Z_f)(\bar{\mu}) = X(\mu)/Z_f$ and $(X/Z_f)(\bar{\gamma}) = X(\gamma)/Z_f$ where the singular sets, if non-empty, are connected, consisting solely of the cores of the filling orbi-solid tori.
Note that these quotients are manifolds exactly when their curves $\bar{\mu}$ and $\bar{\gamma}$ are primitive.

First we observe Lemma~\ref{lem:coverandquotientnew} implies that $\bar{\mu}$ must be a primitive curve because $X/Z_f$ is a quotient of $X$ that covers the orbifold $\mathcal{Z} = X/Z(X)$. Thus we have two cases depending on whether  $\bar{\gamma}$ is a primitive curve.

{\bf Case 1: $\bar{\gamma}$ are primitive.}
  Observe that $\Delta(\bar{\mu},\bar{\gamma}) = |Z_f| \Delta(\mu,\gamma)$ where $\Delta(\bar{\mu},\bar{\gamma})$ is the distance in the boundary of the quotient and $\Delta(\mu,\gamma)$ is measured in $\bdry X$.

  Since $\mu$ and $\gamma$ have primitive images in the quotient manifold $X/Z_f$, $Z_f$ acts freely on the fillings $X(\mu)$ and $X(\gamma)$.  Hence  $(X/Z_f)(\bar{\mu}) = X(\mu)/Z_f$ is a manifold quotient of the lens space $X(\mu)$  and $(X/Z_f)(\bar{\gamma}) = X(\gamma)/Z_f$ is a manifold quotient of $\P=X(\gamma)$.  Thus they are both covered by $S^3$, and therefore they are both finite manifolds (i.e.\ have finite $\pi_1$).   
Since  $X(\gamma)/Z_f$ is covered by $\P$, $|Z_f|$ must be relatively prime to $|\pi_1(\P)| = 120$ by Lemma~\ref{lem:dunbar}(1). Thus we must have $|Z_f|>6$.  Therefore $\Delta(\bar{\mu},\bar{\gamma})>6$ which contradicts the bound on distance between the slopes of two finite surgeries on a hyperbolic knot provided by Boyer and Zhang \cite{boyer1996finite}

{\bf Case 2: $\bar{\gamma}$ is not primitive.}
In this case
$(X/Z_f)(\bar{\gamma})$ is an orbifold that is a cyclic quotient of $\P$ whose singular set is the core of the filling of $X/Z_f$.  However Lemma~\ref{lem:drilling} shows that the exterior of any such core would be the exterior of a torus knot in $S^3$ and hence Seifert fibered, contradicting that $X/Z_f$ must be hyperbolic.
\end{proof}

The final lemma needed to establish the theorem can be proved in a similar manner to the lemma above. 

\begin{lemma}\label{lem:noperiodic}
Let $X$ be the exterior of a hyperbolic knot $K$ in a lens space $L(p,q)$ with $p\geq 2$ admitting a Dehn surgery to $\P$.
Then $|Z(X)|=1$.
\end{lemma}

\begin{proof}
Set $Z=Z(X)$ and assume $|Z| >1$.  
Let $\mu$ be the meridian of $K$ and let $\gamma$ be the slope of the surgery to $\P$.  Then let  $\bar{\mu}$ and $\bar{\gamma}$ be the images of $\mu$ and $\gamma$ in the quotient $\mathcal{Z}=X/Z$.

By Lemma~\ref{lem:noFreeAction}, $\mathcal{Z}$ is an orbifold (with non-empty singular set).
Since $\P$ is a homology sphere, $\mathcal{Z}$ is the exterior of a knot $\bar{K}$ in an orbi-lens space $\mathcal{Z}(\bar{\mu})$ that is disjoint from the singular set  by Lemma~\ref{lem:coverandquotientnew}.  Because $\mathcal{Z}$ is a suborbifold of $\mathcal{Z}(\bar{\mu})$, the singular set  $\Sigma(\mathcal{Z})$ may be identified with the singular set  $\Sigma(\mathcal{Z}(\bar{\mu}))$.  Thus $\Sigma(\mathcal{Z})$ has at most two components.
Since $\P = X(\gamma)$, the quotient orbifold $\P/Z = \mathcal{Z}(\bar{\gamma})$ has a singular set $\Sigma(\mathcal{Z}(\bar{\gamma}))$ consisting of one, two, or three embedded circles according to whether the singular set  $\Sigma(\mathcal{Z})$ has one or two components and whether or not the slope $\bar{\gamma}$ is a primitive curve.

{\bf Case 1: $\bar{\gamma}$ is not primitive.}  
Then $\mathcal{Z}$ is the exterior of one of the components of $\Sigma(\mathcal{Z}(\bar{\gamma}))$.  Hence by Lemmas~\ref{lem:dunbar} and \ref{lem:drilling}, $\mathcal{Z}$ must be Seifert fibered.  Thus $X$ must be Seifert fibered contrary to $K$ being a hyperbolic knot.

{\bf Case 2:  $\bar{\gamma}$ is primitive.} Let $\bar{K}_\gamma$ be the core of the filling solid torus in $\mathcal{Z}(\bar{\gamma})$ so that the exterior of $\bar{K}_\gamma$ is $\mathcal{Z}$.  
Since $\bar{\gamma}$ is primitive, the singular set  $\Sigma=\Sigma(\mathcal{Z}(\bar{\gamma}))$ of $\mathcal{Z}(\bar{\gamma})$ is contained in the knot exterior $\mathcal{Z}$ and thus has either one or two components.  Since $\mathcal{Z}$ is a hyperbolic orbifold in which $\Sigma$ is a geodesic link of one or two components,  the exterior of $\Sigma$ is a hyperbolic manifold $X' = \mathcal{Z}_\Sigma=\mathcal{Z}-\nbhd(\Sigma)$ \cite{kojima1988isometry,sakai1991geodesic}.

Let us write $X'(\alpha,\beta)$ (or $X'(\alpha, \beta_1, \beta_2)$) to denote the Dehn filling of $X'$ along the $\alpha$ and $\beta$ (or $\alpha$ and $\beta_1, \beta_2$) slopes in the components of $\bdry X'$ corresponding to $\bdry \nbhd(\bar{K}_\gamma) = \bdry \mathcal{Z}$ and $\bdry \nbhd(\Sigma)$ respectively.  We will use $-$ to indicate that the boundary component is left unfilled.  

{\bf Case 2a: $\Sigma$ has just one component.}
By Lemma \ref{lem:dunbar}, $X'(\bar{\gamma},-)$ is  a Seifert fibered space over the disk with two exceptional fibers. Thus $X'(\bar{\gamma},-)$ has a one-parameter family of slopes $\sigma_i$ giving lens space fillings $X'(\bar{\gamma},\sigma_i)$. Since $X'(\bar{\mu},-)$ is a solid torus (as it is the complement of the core of an orbi-solid torus in an orbifold lens space with connected singular set),  $X'(\bar{\mu}, \beta_i)$ is also a lens space for each $\beta_i$. Hence, by Thurston's Hyperbolic Dehn Filling Theorem \cite[Theorem 5.8.2]{thurston1979geometry}, we can choose $\beta_i$ so that $X'(-,\beta_i)$ is hyperbolic. However, $\Delta(\bar{\mu}, \bar{\gamma})=|Z| \Delta(\mu,\gamma) \geq 2$, which contradicts the Cyclic Surgery Theorem \cite{CGLS}.

{\bf Case 2b: $\Sigma$ has two components.}
This case is nearly identical to the previous.

By Lemma~\ref{lem:drilling}, $X'(\bar{\gamma},-,-)$ is a Seifert fibered space over the annulus with one exceptional fiber.  Then there are infinitely many pairs of slopes $\{(\sigma_n,\sigma_{n,m}')\}_{(n,m)\in\Z^2}$ such that $X'(\bar{\gamma},\sigma_n,\sigma_{n,m}')$ is a lens space.  Since $X'(\bar{\mu},-,-)$ is a thickened torus (being the complement of the two cores of the orbi-solid tori in the orbi-lens space), $X(\bar{\mu},\sigma_n,\sigma_{n,m}')$ is a lens space for all these pairs of slopes.

Using Thurston's Hyperbolic Dehn Filling Theorem twice, there are infinitely many of these pairs of slopes such that $X(-,\sigma_n,\sigma_{n,m}')$ is hyperbolic.  Again, since $\Delta(\bar{\mu}, \bar{\gamma}) = |Z| \Delta(\mu,\gamma) \geq 2$, we obtain a contradiction to the Cyclic Surgery Theorem \cite{CGLS} for these pairs.
\end{proof}

It may be tempting to try to use $|Z|\cdot \Delta(\mu,\gamma)$ in the previous lemma place try and improve the bound in \cite{boyer1996finite} that $\Delta(\mu,\gamma)\leq 2$. However, this bound is known to be sharp. Also, when $Z(M)$ is trivial, we lose the ability to ``match up'' the cores of the solid tori coming from the Heegaard splitting of the lens space with the exceptional fibers in $\P$, and so the arguments in this appendix will not apply.

\begin{proof}[Proof of Theorem~\ref{thm:symm}]
By Proposition~\ref{prop:noReverse}, $Sym(X)=Sym^+(X)$.  
From Lemma~\ref{lem:coverandquotientnew}(2), $Z(X)$ has index at most $2$ in $Sym(X)$.   Then since $Z(X)=1$ by Lemma~\ref{lem:noperiodic}, we see that $Sym(X)$ is either trivial or $\Z/2\Z$.   
According to Lemma~\ref{lem:coverandquotientnew}(3),
the non-trivial element in this latter case must be a strong involution. 
\end{proof}

\bigskip
\noindent{\bf Acknowledgments.}
The figure in this appendix is courtesy of Ken Baker.

 N.H.\ was supported by Australian Research Council Discovery Grant DP130103694 and by a grant from the Simons Foundation (\#524123 to Neil R. Hoffman). This appendix benefited from the thoughtful suggestions and corrections of the referee. 


\medskip
\myrule{open diamond}{open diamond}
\medskip

\bibliographystyle{plain}
\bibliography{poincareTN2}
\end{document}